\documentclass[12pt]{amsart}
\usepackage{amsmath,amssymb,amsbsy,amsfonts,latexsym,amsopn,amstext,
                                               amsxtra,euscript,amscd,bm,cite}
\usepackage{url}
\usepackage[colorlinks,linkcolor=blue,anchorcolor=blue,citecolor=blue,backref=page]{hyperref}
\usepackage{color}
\usepackage{float}
\usepackage{ mathrsfs }

\hypersetup{breaklinks=true}

\usepackage[norefs,nocites]{refcheck}

\begin{document}

\newtheorem{theorem}{Theorem}
\newtheorem{lemma}[theorem]{Lemma}
\newtheorem{claim}[theorem]{Claim}
\newtheorem{cor}[theorem]{Corollary}
\newtheorem{prop}[theorem]{Proposition}
\newtheorem{definition}{Definition}
\newtheorem{question}[theorem]{Question}
\newtheorem{remark}[theorem]{Remark}
\newcommand{\hh}{{{\mathrm h}}}

\numberwithin{equation}{section}
\numberwithin{theorem}{section}
\numberwithin{table}{section}

\def\sssum{\mathop{\sum\!\sum\!\sum}}
\def\ssum{\mathop{\sum\ldots \sum}}
\def\iint{\mathop{\int\ldots \int}}

\def\squareforqed{\hbox{\rlap{$\sqcap$}$\sqcup$}}
\def\qed{\ifmmode\squareforqed\else{\unskip\nobreak\hfil
\penalty50\hskip1em\null\nobreak\hfil\squareforqed
\parfillskip=0pt\finalhyphendemerits=0\endgraf}\fi}

\newfont{\teneufm}{eufm10}
\newfont{\seveneufm}{eufm7}
\newfont{\fiveeufm}{eufm5}
%
%
\newfam\eufmfam
     \textfont\eufmfam=\teneufm
\scriptfont\eufmfam=\seveneufm
     \scriptscriptfont\eufmfam=\fiveeufm
%
%
\def\frak#1{{\fam\eufmfam\relax#1}}

\newcommand{\Mod}[1]{ \hspace{-2.5mm} \mod{#1}}
\newcommand{\Pmod}[1]{ \hspace{-2.5mm} \pmod{#1}}

\newcommand{\bflambda}{{\boldsymbol{\lambda}}}
\newcommand{\bfmu}{{\boldsymbol{\mu}}}
\newcommand{\bfxi}{{\boldsymbol{\xi}}}
\newcommand{\bfrho}{{\boldsymbol{\rho}}}

\newcommand{\bfalpha}{{\boldsymbol{\alpha}}}
\newcommand{\bfbeta}{{\boldsymbol{\beta}}}
\newcommand{\bfphi}{{\boldsymbol{\varphi}}}
\newcommand{\bfpsi}{{\boldsymbol{\psi}}}
\newcommand{\bftheta}{{\boldsymbol{\vartheta}}}

\def\fK{Frak K}
\def\fT{Frak{T}}

\def\fA{{Frak A}}
\def\fB{{Frak B}}
\def\fC{\mathfrak{C}}

\def \balpha{\bm{\alpha}}
\def \bbeta{\bm{\beta}}
\def \bgamma{\bm{\gamma}}
\def \blambda{\bm{\lambda}}
\def \bchi{\bm{\chi}}
\def \bphi{\bm{\varphi}}
\def \bpsi{\bm{\psi}}

\def\eqref#1{(\ref{#1})}

\def\vec#1{\mathbf{#1}}


\def\cA{{\mathcal A}}
\def\cB{{\mathcal B}}
\def\cC{{\mathcal C}}
\def\cD{{\mathcal D}}
\def\cE{{\mathcal E}}
\def\cF{{\mathcal F}}
\def\cG{{\mathcal G}}
\def\cH{{\mathcal H}}
\def\cI{{\mathcal I}}
\def\cJ{{\mathcal J}}
\def\cK{{\mathcal K}}
\def\cL{{\mathcal L}}
\def\cM{{\mathcal M}}
\def\cN{{\mathcal N}}
\def\cO{{\mathcal O}}
\def\cP{{\mathcal P}}
\def\cQ{{\mathcal Q}}
\def\cR{{\mathcal R}}
\def\cS{{\mathcal S}}
\def\cT{{\mathcal T}}
\def\cU{{\mathcal U}}
\def\cV{{\mathcal V}}
\def\cW{{\mathcal W}}
\def\cX{{\mathcal X}}
\def\cY{{\mathcal Y}}
\def\cZ{{\mathcal Z}}
\newcommand{\rmod}[1]{\: \text{mod} \: #1}

\def\cg{{\mathcal g}}

\def\vr{\mathbf r}

\def\e{{\mathbf{\,e}}}
\def\ep{{\mathbf{\,e}}_p}
\def\em{{\mathbf{\,e}}_m}

\def\Tr{{\mathrm{Tr}}}
\def\Nm{{\mathrm{Nm}\,}}

 \def\SS{{\mathbf{S}}}

\def\lcm{{\mathrm{lcm}}}
\def\ord{{\mathrm{ord}}}

\def\({\left(}
\def\){\right)}
\def\fl#1{\left\lfloor#1\right\rfloor}
\def\rf#1{\left\lceil#1\right\rceil}

\def\mand{\qquad \text{and} \qquad}

\newcommand{\commM}[1]{\marginpar{%
\begin{color}{red}
\vskip-\baselineskip 
\raggedright\footnotesize
\itshape\hrule \smallskip M: #1\par\smallskip\hrule\end{color}}}

\newcommand{\commI}[1]{\marginpar{%
\begin{color}{magenta}
\vskip-\baselineskip 
\raggedright\footnotesize
\itshape\hrule \smallskip I: #1\par\smallskip\hrule\end{color}}}

\newcommand{\commK}[1]{\marginpar{%
\begin{color}{blue}
\vskip-\baselineskip 
\raggedright\footnotesize
\itshape\hrule \smallskip K: #1\par\smallskip\hrule\end{color}}}




\hyphenation{re-pub-lished}

\mathsurround=1pt

\def\bfdefault{b}
\overfullrule=5pt

\def \F{{\mathbb F}}
\def \K{{\mathbb K}}
\def \N{{\mathbb N}}
\def \Z{{\mathbb Z}}
\def \Q{{\mathbb Q}}
\def \R{{\mathbb R}}
\def \C{{\mathbb C}}
\def\Fp{\F_p}
\def \fp{\mathfrak p}
\def \fq{\mathfrak q}

\def\ZK{\Z_K}

\def \xbar{\overline x}
\def\e{{\mathbf{\,e}}}
\def\ep{{\mathbf{\,e}}_p}
\def\eq{{\mathbf{\,e}}_q}


\title[Products in arithmetic progressions]{On products of primes and square-free integers in arithmetic progressions}

\author[Kam Hung Yau]{Kam Hung Yau}

\address{Department of Pure Mathematics, University of New South Wales,
Sydney, NSW 2052, Australia}
\email{kamhung.yau@unsw.edu.au}

\begin{abstract}
We obtain an asymptotic formula for the number of ways to represent every reduced residue class as a product of a prime and square-free integer. This may be considered as a relaxed version of a conjecture of Erd\"os, Odlyzko, and S\'ark\"ozy.
\end{abstract}

\keywords{Kloosterman sums, congruences}
\subjclass[2010]{11L05, 11A07}

\maketitle

\section{Introduction}
 
A conjecture of Erd\"os, Odlyzko, and S\'ark\"ozy~\cite{EOS} asks if for every reduced residue class $a$ modulo $m$ can be represented as a product
\begin{equation} \label{eq: EOS conj}
p_1 p_2 \equiv a \hspace{-3mm} \pmod{m}
\end{equation}
for two primes $p_1, p_2 \le m$. Friedlander, Kurlberg, and Shparlinski~\cite{FKS} considered an average of~\eqref{eq: EOS conj} over $a$ and $m$, and also various modification of~\eqref{eq: EOS conj}. Garaev~\cite{G,G2} improved on these modifications. Other interesting variants of~\eqref{eq: EOS conj} had also been considered by Baker~\cite{B}, Ramar\'e \& Walker~\cite{RW}, Shparlinski~\cite{S,S2}, Walker~\cite{W}.

In this paper, we are concerned with bounding the quantity
$$
\# \left \{(p,s):  ps \equiv a \hspace{-3mm} \pmod{q}, p\le P , s\le S, \mu^2(s)=1, (ps,q)=1 \right  \}
$$
for $(a,q)=1$.
This may also be viewed as a multiplicative analogue in the setting of finite fields of a result of Estermann~\cite{E}.
Estermann~\cite{E} showed that all sufficiently large positive integer can be written as a sum of a prime and a square-free integer, see also~\cite{M1,P1}. Recently, Dudek~\cite{D} showed that this is true for all positive integer greater than two.

Our method uses the nice factoring property of the characteristic function for square-free integers
\begin{equation} \label{eqn: sqfre}
\mu^2(n) = \sum_{d^2|n} \mu(d),
\end{equation}
together with bounds for Kloosterman sums over primes supplied  by Fourvy \& Shparlinski~\cite{FS}, extending those previous result of Garaev~\cite{G}.

\section{Notation}

The notation $U =O(V)$ is equivalent to $U \ll V$ and means there exist an absolute constant $C>0$ such that $U \le CV$.
Exclusively $p$ is a prime number, $\mu$ the M\"obius function, $\tau(n)$ is the number of positive divisors of $n$, and $\varphi(n)$ is the number of positive integer up to $n$ coprime to $n$.

\section{Result}

We denote
\begin{align*}
\pi_q(P) &= \# \{ p\le P : (p,q)=1 \} \\
\intertext{to be the number of primes up to $P$ coprime to $q$, and}
s_q(S) &= \# \{ s \le S : \mu^2(s)=1 , (s,q)=1 \}
\end{align*}
to be the number of square-free integers up to $S$ coprime to $q$.

Finally, for $(a,q)=1$,
denote $\cN_{a,q}^{\#}(P,S)$ by the quantity
\begin{align*}
\# \left \{(p,s):  ps \equiv a \hspace{-3mm} \pmod{q}, p\le P , s\le S, \mu^2(s)=1, (ps,q)=1 \right  \}.
\end{align*}
\begin{theorem} \label{thm: asy bound}
For all fixed $A,\varepsilon >0$, we have
$$
\cN_{a,q}^{\#}(P,S) = \frac{\pi_q(P) s_q(S)}{q} + O \left ( (PS)^{o(1)} S^{1/2} E \right ),
$$
uniformly for $q \le P^{O(1)}$ and  $(a,q)=1$, where
$$
E =
\begin{cases}
\displaystyle P q^{-1}  & \mbox{if $q \le (\log P)^A$,} \\[2pt]
\displaystyle \frac{P }{q^{3/4}}  + \frac{P^{9/10}  }{q^{3/8} }  & \mbox{if $(\log P)^A < q < P^{3/4}$,} \\[10pt]
\displaystyle  \frac{P^{31/32}}{q^{(1-\varepsilon)/2}} + \frac{P^{5/6}}{q^{(3/4-\varepsilon)/2}}    & \mbox{if $P^{3/4} \le q$.}
\end{cases}
$$
\end{theorem}

The main term in Theorem~\ref{thm: asy bound} is
\begin{align*}
\frac{\pi_q(P) s_q(S)}{q}  & \gg \frac{1}{q}  \frac{P}{\log P} \left ( \frac{\varphi(q) S }{q} + O(\tau(q) ) \right ) \\
& \gg P^{1+o(1)}  Sq^{-1} 
\end{align*}
since $q \le P^{O(1)}$. It follows that $\cN_{a,q}^{\#}(P,S)  >0$ when $P \rightarrow \infty$ if either one of the following three conditions below holds.
\begin{enumerate}
\item  $q \le (\log P)^A$ and there exist $\varepsilon>0$ such that $S \gg P^{\varepsilon} $. \\
\item   $(\log P)^A < q < P^{3/4}$ and there exist $\varepsilon>0$ such that 
$$
S^{2} \gg (PS)^{\varepsilon} q \mand  P^{4 } S^{20} \gg (PS)^{\varepsilon} q^{25}. \\
$$
\item  $P^{3/4} \le q$ and there an $\varepsilon >0$ such that 
$$
P S^{16} \gg (PS)^{\varepsilon }q^{16}    \mand   P^{4 } S^{12 } \gg (PS)^{\varepsilon} q^{15}.
$$
\end{enumerate}

\section{Preliminaries}

For $(a,q)=1$, we denote the Kloosterman sum over primes
$$
S_q(a;x) = \sum_{\substack{p \le x \\ (p,q)=1}} \mathbf{e}_q(a\bar{p} ).
$$
Here $\mathbf{e}_q(x)=\exp(2\pi i x/q)$ and $\overline p$ is the multiplicative inverse for $p$ modulo $q$.
Bounds for $q$ prime had been obtained by Garaev~\cite{G}. Fouvry and Shparlinski~\cite{FS} extended these results for composite $q$.
We gather Theorem 3.1, 3.2  and Equation~(3.13) from~\cite{FS} into the following lemma.
\begin{lemma} \label{lem: exp bound}
For every fixed $A, \varepsilon >0$, we have
$$
S_q(a;x)  = O(  B_q(x) ),
$$
uniformly for integer $q \ge 2$, $(a,q)=1$, and $x \ge 2$.
Here
$$
B_q(x) =
\begin{cases}
\displaystyle x^{1+o(1)}q^{-1}   & \mbox{if $q \le (\log x)^A$,} \\[5pt]
\displaystyle   (q^{-1/2} x +  q^{1/4} x^{4/5})x^{o(1)}   & \mbox{if $(\log x)^A < q < x^{3/4}$,} \\[5pt]
\displaystyle  (x^{15/16} + q^{1/4} x^{2/3} )q^{\varepsilon}  & \mbox{if $x^{3/4} \le q$.} \\
\end{cases}
$$
\end{lemma}

Denote
$$
\cN_{a,q}(P,S)=\# \left \{(p,s):  ps \equiv a \hspace{-3mm} \pmod{q}, p\le P , s\le S, (ps,q)=1 \right  \}
$$
for $(a,q)=1$.
Below, we provide an upper bounds for $\cN_{a,q}(P,S)$.

\begin{lemma} \label{lem: upper bound for cN}
For $q \le P^{O(1)}$, we have
$$
\cN_{a,q}(P,S) \ll \left (\frac{PS}{q} + 1 \right )(PS)^{o(1)}.
$$
\end{lemma}

\begin{proof}
We count the number of solution to
$
ps = a + kq.
$
Therefore we bound $k \ll (PS/q+1)$. For each $a+kq$, the number of distinct prime factor is no more than
$$
\ll \log(kq) \ll \log (PS+q) \ll \log (PS) \ll (PS)^{o(1)},
$$
from our upper bound on $k$.
\end{proof}

Denote 
$$
N_q(P,S) = \# \{ (p,s) : p \le P, s \le S , (ps,q)=1 \}.
$$
We relate the quantity $\cN_{a,q}(P,S)$ with $N_q(P,S)$.

\begin{lemma} \label{lem: asy for cN}
For all fixed  $\varepsilon>0$, we have
$$
\cN_{a,q}(P,S)  = \frac{N_q(P,S)}{q}   +   O(B_q(P)),
$$
uniformly for $(a,q)=1$, where $B_q$ is defined as in Lemma~\normalfont{\ref{lem: exp bound}}.
\end{lemma}

\begin{proof}
We interpret this as a uniform distribution problem. Namely we consider
$$
s \equiv a \bar{p}  \hspace{-2mm} \pmod{q}
$$
which fall in the interval $[1,S]$. The result follows from Lemma~\ref{lem: exp bound} applied with the Erd\"os-Tur\'an inequality, see~\cite{DT}.
\end{proof}

We provide a bound for $N_q(P,S)$.

\begin{lemma} \label{lem: N_q bound}
For $q \le P^{O(1)}$, we have
$$
N_q(P,S) = \frac{\varphi(q) \pi_q(P) S}{q} +O(P^{1+o(1) }).
$$
\end{lemma}

\begin{proof}
Note the identity
$$
\sum_{d|n} \mu(d) =
\begin{cases}
1 &\mbox{if $n=1$,} \\
0 & \mbox{otherwise.}
\end{cases}
$$
We have
\begin{align*}
N_q(P,S) & = \sum_{\substack{p \le P \\ (p,q)=1 }}1 \sum_{\substack{s \le S \\ (s,q)=1 }}1   \\
& =\pi_q(P)  \sum_{s \le S} \sum_{\substack{d|s \\ d|q}} \mu(d) \\
&= \pi_q(P) \left (\frac{\varphi(q) S}{q} + O(\tau(q))\right  ) \\
& = \frac{\varphi(q) \pi_q(P) S}{q} +O(P^{1+o(1) }).
\end{align*}
\end{proof}

We also provide a bound for $s_q(S)$.

\begin{lemma} \label{lem: asy for s_q}
We have
$$
s_q(S) = \frac{\varphi(q)}{q} \prod_{p\nmid q} \left (1-\frac{1}{p^2} \right )  S +O(  S^{1/2} q^{o(1)}).
$$
\end{lemma}

\begin{proof}
In a first step
\begin{align*}
s_q(S) & = \sum_{ \substack{d \le S^{1/2} \\ (d,q)=1 }} \mu(d) \sum_{\substack{s \le S/d^2 \\ (s,q)=1 }} 1 \\
& = \sum_{ \substack{d \le S^{1/2} \\ (d,q)=1 }} \mu(d) \sum_{\substack{s \le S/d^2  }} \sum_{\substack{r|s \\r|q}} \mu(r). 
\end{align*}
Interchanging summation and completing the series, we get
\begin{align*}
s_q(S) & = \sum_{r|q} \mu(r) \sum_{ \substack{d \le S^{1/2} \\ (d,q)=1 }} \mu(d) \left ( \frac{S}{d^2r} +O(1)  \right ) \\
& = \frac{\varphi(q)}{q}  \left ( \sum_{\substack{d=1 \\ (d,q)=1 }}^{\infty}  \frac{\mu(d) }{d^2} - \sum_{\substack{d>S^{1/2} \\ (d,q)=1 }} \frac{\mu(d) }{d^2} \right ) S +O(S^{1/2} \tau(q)) \\
& = \frac{\varphi(q)}{q}  \prod_{p\nmid q}\left (1-\frac{1}{p^2} \right ) S +O(  S^{1/2} q^{o(1)} ),
\end{align*}
by noting that
$$
\varphi(q) = q\prod_{p|q} \left (1-\frac{1}{q} \right ) = q \sum_{r|q}\frac{\mu(r)}{r}.
$$
\end{proof}

\section{Proof of Theorem~\ref{thm: asy bound}}

Using~\eqref{eqn: sqfre}, we obtain
\begin{align*}
\cN_{a,q}^{\#}(P,S) & = \mathop{ \sum_{p \le P}  \sum_{s \le S}  }_{ps \equiv a \hspace{-2mm} \pmod{q}, \hspace{0.5mm} (ps,q)=1    } \mu^2(s)  \\
& = \sum_{\substack{d\le S^{1/2} \\ (d,q)=1 }} \mu(d)  \cN_{ad^{-2},q}(P,S/d^2) \\
& = \Sigma_1 + \Sigma_2,
\end{align*}
where
\begin{align*}
\Sigma_1 & = \sum_{\substack{ d \le D \\  (d,q)=1 }} \mu(d)  \cN_{ad^{-2},q}(P,S/d^2), \\
\intertext{and}
\Sigma_2 & = \sum_{ \substack{ D <d\le S^{1/2} \\ (d,q)=1  } } \mu(d)  \cN_{ad^{-2},q}(P,S/d^2).
\end{align*}
Here $D=D(P,S)$ is a parameter that will be chosen later.

By Lemma~\ref{lem: upper bound for cN}, we bound
\begin{align*}
\Sigma_2  & \ll \sum_{D <d\le S^{1/2} } \left  (\frac{PS}{d^2 q} +1  \right  ) \left (\frac{PS}{d^2} \right  )^{o(1)} \\
& \ll (PS )^{o(1)} \left   ( \frac{PS}{qD} + S^{1/2}   \right ).
\end{align*}

By Lemma~\ref{lem: asy for cN} and~\ref{lem: N_q bound} we get
\begin{align*}
\Sigma_1 & = \sum_{\substack{ d \le D \\  (d,q)=1 }} \mu(d) \left (\frac{N_q(P,S/d^2)}{q} + O(B_q(P)) \right ) \nonumber \\ 
& = \sum_{\substack{d \le D \\ (d,q)=1}} \mu(d) \left (\frac{\varphi(q) \pi_q(P) S }{q^2d^2}  + O(P^{1+o(1)}q^{-1}) \right )  + O(DB_q(P)). \nonumber \\
\end{align*}

Completing the series in the summation over $d$, we assert
\begin{align}
\Sigma_1 & = \frac{\varphi(q) \pi_q(P) S }{q^2} \left (  \sum_{\substack{d =1  \\ (d,q)=1}}^{\infty} \frac{\mu(d)}{d^2} - \sum_{\substack{d > D \\ (d,q)=1}} \frac{\mu(d)}{d^2} \right ) \nonumber \\
&\quad  + O(D \{B_q(P) + P^{1+o(1)}q^{-1} \} ) \nonumber \\
& = \frac{ \pi_q(P) }{q}  \left ( \frac{\varphi(q)S}{q} \sum_{\substack{d =1  \\ (d,q)=1}}^{\infty} \frac{\mu(d)}{d^2}  \right )+O \left ( \frac{PS}{qD} + DB_q(P)  \right  )  \nonumber \\
& = \frac{\pi_q(P) s_q(S)}{q} +O \left ( \frac{S^{1/2} \pi_q(P) }{q^{1+o(1)}}  + \frac{PS}{qD} + DB_q(P)  \right  ), \label{eq: error term}
\end{align}
where the last line follows from applying Lemma~\ref{lem: asy for s_q}.

Now we set 
$$
D = 
\begin{cases}
\displaystyle  S^{1/2}P^{o(1)}   & \mbox{if $q \le (\log P)^A$,}\\
\displaystyle \left (\frac{PS}{Pq^{1/2} + q^{5/4}P^{4/5}} \right )^{1/2}P^{o(1)} & \mbox{if $(\log P)^A < q < P^{3/4}$,} \\
\displaystyle  \left (\frac{PS}{q^{1+\varepsilon}(P^{15/16} + q^{1/4}P^{2/3})} \right )^{1/2} & \mbox{if $P^{3/4} \le q$.}
\end{cases}
$$
Then the last two terms in~\eqref{eq: error term} are equal and it follows
$$
\cN_{a,q}^{\#}(P,S) = \frac{\pi_q(P) s_q(S)}{q} +O \left ( \left  (  \frac{S^{1/2} \pi_q(P) }{q^{1+o(1)}}  +\frac{PS}{qD} + S^{1/2} \right ) (PS )^{o(1)}      \right   ).
$$
If $q\le (\log P)^{A}$ then the error term above is majorised by 
\begin{align*}
&    \left ( \frac{PS^{1/2}}{q} + S^{1/2} \right ) (PS)^{o(1)}   \\
& \ll   PS^{1/2}q^{-1} (PS)^{o(1)}.
\end{align*}
If $(\log P)^A < q < P^{3/4}$ then the error term above is majorised by 
\begin{align*}
&    \left( \frac{P^{1/2} S^{1/2} (Pq^{1/2} + q^{5/4} P^{4/5} )^{1/2} }{q}    +S^{1/2}         \right )(PS)^{o(1)}   \\
& \ll S^{1/2} \left (\frac{P }{q^{3/4}}  + \frac{P^{9/10} }{q^{3/8} } \right )(PS)^{o(1)}.
\end{align*}
Lastly, if $P^{3/4} \le q$ then the error term above is majorised by 
\begin{align*}
&  \left ( \frac{P^{1/2} S^{1/2} (q^{1+\varepsilon} \{ P^{15/16} + q^{1/4} P^{2/3}  \} )^{1/2} }{q} +S^{1/2} \right )(PS)^{o(1)} \\
& \ll S^{1/2}\left ( \frac{P^{31/32}}{q^{(1-\varepsilon)/2}} + \frac{P^{5/6}}{q^{(3/4-\varepsilon)/2}}  \right )(PS)^{o(1)}.
\end{align*}
The result follows.

\section*{Acknowledgement}

The author thanks I. E. Shparlinski for the problem and helpful comments together with Liangyi Zhao. The author also thanks the referee for helpful comments. This work is supported by an Australian Government Research Training Program (RTP) Scholarship.


\begin{thebibliography}{12}


\bibitem{B} R. C. Baker, {\it Kloosterman sums with prime variable.\/} Acta Arith. {\bf 156} (4) (2012), 351--372. 

\bibitem{DT} M. Drmota, R. Tichy, {\it Sequences, discrepancies and applications.\/} SpringerVerlag, Berlin, 1997.


\bibitem{D} A. W. Dudek, {\it On the sum of a prime and a square-free number.\/} Ramanujan J. {\bf 42} (1) (2017), 233--240. 

\bibitem{EOS} P. Erd\"os, A. M. Odlyzko, A. S\'ark\"ozy, {\it On the residues of products of prime numbers.\/} Period. Math. Hungar. {\bf 18} (3) (1987), 229--239. 

\bibitem{E} T. Estermann, {\it On the Representations of a number as the sum of a prime and a quadratfrei number.\/} J. London Math. Soc. {\bf 6} (3) (1931), 219--221. 


\bibitem{FKS} J. Friedlander, P. Kurlberg, I. E. Shparlinski, {\it Products in residue classes.\/} Math. Res. Lett. {\bf 15} (6) (2008), 1133--1147.

\bibitem{FS} E. Fouvry, I. E. Shparlinski, {\it On a ternary quadratic form over primes.\/} Acta Arith. {\bf 150} (3) (2011), 285--314.

\bibitem{G} M. Z. Garaev, {\it An estimate of Kloosterman sums with prime numbers and application.\/} Mat. Zametki {\bf 88} (2010), 365--373 (in Russian).


\bibitem{G2} M. Z. Garaev, {\it On multiplicative congruences.\/} Math. Z. {\bf 272} (1-2) (2012), 473--482. 

\bibitem{M1} L. Mirsky, {\it The number of representations of an integer as the sum of a prime and a $k$-free integer.} Amer. Math. Monthly {\bf 56} (1949), 17--19.


\bibitem{P1} A. Page, {\it On the number of primes in an arithmetic progression.} Proc. London Math. Soc. (2) {\bf 39} (2) (1935), 116--141.

\bibitem{RW} O. Ramar\'e, A. Walker, {\it Products of primes in arithmetic progressions: a footnote in parity breaking.\/}
J. Th\'eor. Nombres Bordeaux {\bf 30} (1) (2018), 219--225. 

\bibitem{S} I. E. Shparlinski, {\it On products of primes and almost primes in arithmetic progressions.\/} Period. Math. Hungar. {\bf 67} (1) (2013), 55--61.

\bibitem{S2} I. E. Shparlinski, {\it On short products of primes in arithmetic progressions.\/} Proc. Amer. Math. Soc. {\bf 147} (3) (2019), 977--986. 

\bibitem{W} A. Walker, {\it A multiplicative analogue of Schnirelmann's theorem.\/} Bull. Lond. Math. Soc. {\bf 48} (6) (2016), 1018--1028.


\end{thebibliography}
\end{document}